\documentclass{svproc}
\usepackage{url}

\usepackage{hyperref}
\usepackage{type1cm}        % activate if the above 3 fonts are
\usepackage{makeidx}         % allows index generation
\usepackage{graphicx}        % standard LaTeX graphics tool
\usepackage{multicol}        % used for the two-column index
\usepackage[bottom]{footmisc}% places footnotes at page bottom

\usepackage{newtxtext}       %
\usepackage[varvw]{newtxmath}       % selects Times Roman as basic font

\usepackage{bm}% for bold math symbols
\usepackage{siunitx}% for table alignments
\usepackage{bbm}

\newcommand{\bsa}{{\boldsymbol{a}}}
\newcommand{\bsb}{{\boldsymbol{b}}}

\newcommand{\bst}{{\boldsymbol{t}}}

\newcommand{\bsw}{{\boldsymbol{w}}}
\newcommand{\bsx}{{\boldsymbol{x}}}
\newcommand{\bsy}{{\boldsymbol{y}}}
\newcommand{\bsz}{{\boldsymbol{z}}}

\newcommand{\bsT}{{\boldsymbol{T}}}

\newcommand{\bszero}{{\boldsymbol{0}}} % vector of zeros
\newcommand{\bsone}{{\boldsymbol{1}}}  % vector of ones
\newcommand{\bsgamma}{{\boldsymbol{\gamma}}}

\newcommand{\bstau}{{\boldsymbol{\tau}}}

\newcommand{\rd}{{\mathrm{d}}}

\newcommand{\bbE}{{\mathbb{E}}}

\newcommand{\bbN}{{\mathbb{N}}}

\newcommand{\bbP}{{\mathbb{P}}}

\newcommand{\bbR}{{\mathbb{R}}}

\DeclareSymbolFont{bbold}{U}{bbold}{m}{n}
\DeclareSymbolFontAlphabet{\mathbbold}{bbold}

\newcommand{\calH}{{\mathcal{H}}}

\newcommand{\calK}{{\mathcal{K}}}

\newcommand{\setu}{{\mathfrak{u}}}
\newcommand{\setv}{{\mathfrak{v}}}

\newcommand{\KaKlScSuKME}{\mathsf{KME}}
\newcommand{\KaKlScSuquark}{\setbox0\hbox{$x$}\hbox to\wd0{\hss$\cdot$\hss}}

\begin{document}
	
\mainmatter              % start of a contribution
\title{Lattice Rules Meet Kernel Cubature}
\titlerunning{Lattice Rules Meet Kernel Cubature}  % abbreviated title (for running head)
\author{Vesa Kaarnioja\inst{1,2} \and Ilja Klebanov\inst{2}\and
	Claudia Schillings\inst{2}  \and Yuya Suzuki\inst{3} }
\authorrunning{Vesa Kaarnioja et al.} % abbreviated author list (for running head)
\tocauthor{Vesa Kaarnioja, Ilja Klebanov, Claudia Schillings, and Yuya Suzuki}
\institute{School of Engineering Sciences, LUT University, P.O.~Box 20, 53851 Lappeenranta, Finland,\\
	{\tt vesa.kaarnioja@iki.fi}
	\and
	Department of Mathematics and Computer Science, Free University of Berlin, Arnimallee 6, 14195 Berlin, Germany,\\
	{\tt klebanov@zedat.fu-berlin.de, c.schillings@fu-berlin.de}
	\and
	Department of Mathematics and Systems Analysis, Aalto University, P.O.~Box 11100, 00076 Aalto, Finland,\\
	{\tt yuya.suzuki@aalto.fi}
}

\maketitle              % typeset the title of the contribution

\begin{abstract}
	Rank-1 lattice rules are a class of equally weighted quasi-Monte Carlo methods that achieve essentially linear convergence rates for functions in a reproducing kernel Hilbert space (RKHS) characterized by square-integrable first-order mixed partial derivatives.
	In this work, we explore the impact of replacing the equal weights in lattice rules with optimized cubature weights derived using the reproducing kernel.
	We establish a theoretical result demonstrating a doubled convergence rate in the one-dimensional case and provide numerical investigations of convergence rates in higher dimensions.
	We also present numerical results for an uncertainty quantification problem involving an elliptic partial differential equation with a random coefficient.
\end{abstract}

\section{Introduction}
Computing the expected value $\mathbb E_{\mathbb P}[f]$ of a function $f\colon D \to \mathbb R$ over a domain $D \subseteq \mathbb R^s$ with respect to a probability distribution $\mathbb P$ is a fundamental problem in fields such as uncertainty quantification, machine learning, statistics, financial mathematics, and statistical mechanics. Since these integrals are often intractable analytically, they are approximated numerically using an empirical mean:
\begin{equation}
	\label{KaKlScSu:equ:empirical_mean_approximation}
	\mathbb E_{\mathbb P}[f]
	=
	\int_D f(x)\,{\rm d}\mathbb P(x)
	\approx
	\sum_{k=0}^{n-1} w_k f(\bst_k)
	=
	\bbE_{\mathbb P_{\bsT}^{\boldsymbol w}}[f],
	\qquad
	\bbP_{\bsT}^{\bsw} := \sum_{k=0}^{n-1} w_{k} \delta_{\bst_{k}}.
\end{equation}
The central challenge in constructing higher-order cubature methods lies in the careful selection of evaluation points $\bsT = (\bst_k)_{k=0}^{n-1} \in D^n$ and weights $\bsw = (w_k)_{k=0}^{n-1} \in \bbR^n$ to ensure favorable approximation properties of the error $|\mathbb E_{\mathbb P}[f] - \mathbb E_{\mathbb P_{\bsT}^{\boldsymbol w}}[f]|$.
% Here, $\bbP_{\bsT}^{\bsw} \coloneqq \sum_{k=0}^{n-1} w_{k} \delta_{\bst_{k}}$.

To address this challenge for potentially high-dimensional integration problems, one can either use sampling-based methods or numerical cubature rules. Sampling-based approaches include methods like Markov chain Monte Carlo (MCMC), which construct a Markov chain with the target distribution $\mathbb{P}$ as its stationary distribution, and importance sampling, which modifies the probability measure to reduce variance and often enables direct sampling.
In contrast, numerical cubature methods such as sparse grids and quasi-Monte Carlo (QMC) methods construct the nodes $\bsT$ and weights $\bsw$ deterministically and can achieve faster convergence rates under sufficient smoothness assumptions on the integrand.
While QMC is fundamentally deterministic, in our numerical experiments (Section~\ref{KaKlScSu:sec:numerics}) we use a randomized QMC method based on several \emph{random shifts} of a fixed lattice which is common practice in QMC.
We focus on a particularly simple QMC rule in which $\bsT$ is chosen to be a lattice, as introduced in Section~\ref{KaKlScSu:subsec:lattice_rules_tent_transform}.

In this work, we consider cubature rules with a known convergence rate in a reproducing kernel Hilbert space (RKHS) $\mathcal{H}$. We interpret the cubature rule as an element of the subspace $V_{\bsT} = \mathrm{span}(K(\bst_0,\KaKlScSuquark), \ldots, K(\bst_{n-1},\KaKlScSuquark)) \subseteq \mathcal{H}$, spanned by the reproducing kernel with its first argument fixed at the cubature nodes $\bst_k$.

Specifically, we focus on the \emph{kernel mean embeddings} of $\mathbb{P}$ and $\mathbb{P}_{\bsT}^{\boldsymbol{w}^{\ast}}$:
\[
h
=
\int_D K(\bsx, \KaKlScSuquark) \, \mathrm{d}\mathbb{P}(\bsx) \in \mathcal{H},
\qquad
h_{\bsT}^{\bsw^{\ast}}
=
\sum_{k=0}^{n-1} w_k^* K(\bst_k, \KaKlScSuquark) \in V_{\bsT} \subseteq \mathcal{H},
\]
and choose the weights $\bsw^*$ so that $h_{\bsT}^{\bsw^{\ast}}$ becomes the $\mathcal{H}$-orthogonal projection (or equivalently, the kernel interpolant, cf.\ Lemma \ref{KaKlScSu:lemma:orthogonal_projection_coincides_with_interpolation_in_RKHS}) of $h$ onto $V_{\bsT}$.
We refer to the resulting cubature rule $Q_{\bsT}^{\bsw^{\ast}} f := \sum_{k=0}^{n-1} w_{k}^{\ast} f(\bst_{k})$ as \emph{kernel cubature}.
By the reproducing-type properties (cf.\ \eqref{KaKlScSu:equ:KME_reproducing_property} below) of $h$ and $h_{\bsT}^{\bsw^{\ast}}$,
\[
\mathbb{E}_{\mathbb{P}}[f] = \langle f, h \rangle_{\mathcal{H}},
\qquad
\mathbb{E}_{\mathbb{P}_{\bsT}^{\bsw^{\ast}}}[f] = \langle f, h_{\bsT}^{\bsw^{\ast}} \rangle_{\mathcal{H}},
\]
the choice of $\bsw^*$ ensures that these two expected values are close, yielding a cubature rule as in \eqref{KaKlScSu:equ:empirical_mean_approximation} with favorable approximation properties.
In fact, it guarantees that the worst-case error
\begin{equation}
	\label{KaKlScSu:equ:WCE_general_form_in_RKHS}
	e(Q_{\bsT}^{\bsw} , \calH)
	:=
	\sup_{\| f \|_{\calH}=1} | \mathbb{E}_{\mathbb P}[f] - Q_{\bsT}^{\bsw} f|
	=
	\sup_{\| f \|_{\calH}=1} \langle h - h_{\bsT}^{\bsw} , f \rangle_{\calH}
	=
	\| h - h_{\bsT}^{\bsw} \|_{\calH}
\end{equation}
is minimized by $\bsw^{\ast}$, where
$Q_{\bsT}^{\bsw} f := \bbE_{\bbP_{\bsT}^{\bsw}}[f]$.
While conventional QMC analysis primarily focuses on the worst-case error, we aim to achieve an additional improvement in the convergence of the approximation error, motivated by the observation (cf.\ Proposition~\ref{KaKlScSu:prop:approximation_error_additional_gain}) that, for a fixed $f \in \calH$,
\begin{equation}
	\label{KaKlScSu:equ:approximation_error_additional_gain}
	|\mathbb{E}_{\mathbb{P}}[f] - Q_{\bsT}^{\bsw^{\ast}}f|
	\leq
	e(Q_{\bsT}^{\bsw^{\ast}}, \calH)\,
	\mathrm{dist}_{\calH}(f,V_{\bsT}),
\end{equation}
a result that is unique to the optimally weighted cubature rule $Q_{\bsT}^{\bsw^{\ast}}$.
This expectation arises from the intuition that the distance $\mathrm{dist}_{\calH}(f,V_{\bsT}) = \inf_{v \in V_{\bsT}} \|{f-v}\|_{\calH}$ between $f$ and $V_{\bsT}$ should decrease as $n$ increases and $V_{\bsT}$ increasingly approximates $\calH$.
However, we cannot yet establish a precise rate for this decay.

The optimal weights $\bsw^{\ast}$ coincide with Bayesian cubature weights for the canonical choice of the prior
\cite{KaKlScSu:Briol2019probabilistic,KaKlScSu:Diaconis1988Bayesian,KaKlScSu:Ghahramani2002Bayesian,KaKlScSu:OHagan1991Bayes,KaKlScSu:Ritter2000average}.
The recent work of Hickernell and Jagadeeswaran \cite{KaKlScSu:Jagadeeswaran2019BayesianCubatureMatching,KaKlScSu:Jagadeeswaran2019BayesianCubatureLattice,KaKlScSu:Jagadeeswaran2022BayesianCubatureSobol} has investigated the construction of optimized cubature weights for fixed sequences of lattice points and Sobol$'$ nets, but only within the context of shift-invariant kernels and Walsh kernels.
We note that our construction is related to recent studies on kernel interpolation over lattice point sets~\cite{KaKlScSu:KKKNS2022,KaKlScSu:kks_serendip,KaKlScSu:sk23} in the sense that the kernel cubature of a function $f \in \calH$ is equivalent to computing the integral of its kernel interpolant over the cubature point set. However, our work addresses the non-periodic setting, while the works~\cite{KaKlScSu:KKKNS2022,KaKlScSu:kks_serendip,KaKlScSu:sk23} only discuss kernel interpolation over lattice point sets in the periodic setting.

\vspace{2ex}
\textbf{Contributions.}
We make the following contributions to weighted QMC methods:
\begin{enumerate}
	\item 
	We propose a weighted version of QMC cubature, where the weights minimize the distance between the kernel mean embeddings of the empirical and true distribution, leading to low approximation error, as observed in various numerical experiments across low and high dimensions.
	
	\item 
	While, by construction, the worst-case error $e(Q_{\bsT}^{\bsw^{\ast}}, \calH)$ is minimal among all possible weights $\bsw$, it shows only a slight improvement over the equally weighted case.  
	However, we observe a significant reduction in the approximation error $|\mathbb{E}_{\mathbb{P}}[f] - \mathbb{E}_{\mathbb{P}_{\bsT}^{\boldsymbol{w}}}[f]|$, which we attribute to the factor $\mathrm{dist}_{\calH}(f,V_{\bsT})$ in the bound \eqref{KaKlScSu:equ:approximation_error_additional_gain}.  
	Unlike the constant factor $\| f \|_{\calH}$ in the classical bound for equally weighted QMC, $\mathrm{dist}_{\calH}(f,V_{\bsT})$ can be expected to decrease as $n$ increases.
	
	\item 
	Although the explicit error convergence rates for optimally weighted lattice point sets remain an open problem in higher dimensions, we prove in the one-dimensional case that this approach leads to a doubled rate of convergence compared to the equally weighted case.

	\item 
	We numerically investigate the behavior of the worst-case error in Sobolev spaces of higher smoothness $\alpha = 4$ using weights optimized for the less smooth setting $\alpha = 2$.
	A significant improvement in convergence speed is observed.
	This experiment is conducted in low dimensions only ($s=2$) with tent-transformed lattices, which are known to enhance the convergence rate from first to second order.
\end{enumerate}

% \vspace{2ex}
\textbf{Outline.}
This document is structured as follows.
After introducing our setup and notation in Section~\ref{KaKlScSu:sec:notation}, we provide our theoretical contributions in Section~\ref{KaKlScSu:sec:theory}.
In Section~\ref{KaKlScSu:sec:numerics} we numerically demonstrate the improvement of kernel cubature over equally weighted lattice rules and provide a conclusion in Section~\ref{KaKlScSu:sec:conclusion}.

\section{Preliminaries and Notation}
\label{KaKlScSu:sec:notation}
Throughout this manuscript, we will use the following general notation:
$D\subseteq \bbR^{s}$, $s\in\bbN$, will be the domain of interest equipped with its Borel $\sigma$-algebra and a probability measure $\bbP$, typically $D = [0,1]^{s}$ with uniform measure $\bbP = \mathsf{Unif}_{D}$.
We denote by $\{ v \} := (v_{j} - \lfloor v_{j} \rfloor)_{j=1,\dots,s}$ the componentwise fractional part of a vector $v \in \mathbb{R}^{s}$, by $\mathbf 1 = (1)_{k=0}^{n-1}$ the $n$-dimensional unit vector, and by $\mathbbm{1} \colon D \to \bbR$ the constant unit function.
Further, for a function $f \in \calH$ we denote
\begin{align*}
	If
	&=
	\int_{D} f \, \mathrm{d} \bbP = \bbE_{\bbP}[f],
	&&
	\\
	Q_{\bsT} f
	&=
	n^{-1} \sum_{k=0}^{n-1} f(\bst_{k}) = \bbE_{\bbP_{\bsT}}[f],
	&
	\bbP_{\bsT}
	&=
	n^{-1} \sum_{k=0}^{n-1} \delta_{\bst_{k}},
	\\
	Q_{\bsT}^{\bsw} f
	&=
	\sum_{k=0}^{n-1} w_{k} f(\bst_{k}) = \bbE_{\bbP_{\bsT}^{\bsw}}[f],
	&
	\bbP_{\bsT}^{\bsw}
	&=
	\sum_{k=0}^{n-1} w_{k} \delta_{\bst_{k}},
\end{align*}
where the cubature rules $Q_{\bsT}$ and $Q_{\bsT}^{\bsw}$ are based on the \emph{evaluation points} $\bsT = (\bst_{k})_{k=0}^{n-1} \in D^{n}$, which in this work will be a (potentially shifted and tent-transformed, cf.\ Section~\ref{KaKlScSu:subsec:lattice_rules_tent_transform}) lattice, and \emph{cubature weights}, $\bsw = (w_{k})_{k=0}^{n-1} \in \bbR^{n}$.
Here, $\bbP_{\bsT}$ and $\bbP_{\bsT}^{\bsw}$ denote the corresponding (possibly signed) discrete measures on $D$.
% Note that we make no assumptions on the cubature weights to be non-negative or to sum to one, hence these measures, while being finite, may fail to be probability measures and can attain negative values.
Note that we make no assumptions on the cubature weights to be non-negative or to sum to one. This is in line with common practice in Bayesian cubature, where the primary objective is optimization rather than strictly enforcing a probabilistic interpretation of the weights.
Consequently, the resulting measures, while being finite, may fail to be probability measures, can attain negative values, and may introduce a bias.

\subsection{Reproducing Kernel Hilbert Spaces and Kernel Mean Embeddings}

The (possibly signed) measures $\bbP,\bbP_{\bsT},\bbP_{\bsT}^{\bsw}$ will be \emph{embedded} into a reproducing kernel Hilbert space (RKHS; \cite{KaKlScSu:berlinet2004rkhs}) $\calH$ corresponding to a symmetric and positive definite kernel
$K \colon D\times D \to \bbR$.
Note that we work with \emph{strictly} positive definite kernels rather than semi-positive definite ones in the sense that the \emph{Gram matrix} $G=(K(\bsx_i,\bsx_j))_{i,j=1}^N$ is (\emph{strictly}) positive definite, and thereby invertible, for all $N\in\mathbb N$ and pairwise distinct $\bsx_i\in D$, $i=1,\dots,N$.
For a signed measure $\mu$ on $D$ its kernel mean embedding ($\KaKlScSuKME$) is defined by
\[
\KaKlScSuKME(\mu)
:=
\int_{D} K(\bsx,\KaKlScSuquark) \, \mathrm{d} \mu(\bsx)
\in
\calH.
\]
Strictly speaking, the KME is defined only for certain combinations of kernels and signed measures \cite{KaKlScSu:berlinet2004rkhs}, in particular, the corresponding integral must be well defined.
We omit these technical details here, as the assumptions are always fulfilled for the measures and kernels considered in this paper.
Importantly, the KME satisfies a reproducing-type property \cite{KaKlScSu:smola2007embedding}
\begin{equation}
	\label{KaKlScSu:equ:KME_reproducing_property}    
	\langle \KaKlScSuKME(\mu) , f \rangle_{\calH}
	=
	\bbE_{\mu}[f],
	\qquad
	f \in \calH.
\end{equation}
After defining
$V_{\bsT} = \mathrm{span}(K(\bst_0,\KaKlScSuquark), \ldots, K(\bst_{n-1},\KaKlScSuquark)) \subseteq \mathcal{H}$
and denoting by $P_{V_{\bsT}}\colon \calH \to V_{\bsT}$ the corresponding $\calH$-orthogonal projection, the following embeddings will be crucial:
\begin{align*}
	h
	&=
	\KaKlScSuKME(\bbP)
	=
	\int_{D} K(\bsx,\KaKlScSuquark) \bbP(\mathrm d \bsx) \in \calH,
	\\
	h_{\bsT}
	&=
	\KaKlScSuKME(\bbP_{\bsT})
	=
	\int_{D} K(\bsx,\KaKlScSuquark) \bbP_{\bsT}(\mathrm d\bsx)
	=
	n^{-1} \sum_{k=0}^{n-1} K(\bst_{k},\KaKlScSuquark) \in V_{\bsT} \subseteq \calH,
	\\
	h_{\bsT}^{\bsw}
	&=
	\KaKlScSuKME(\bbP_{\bsT}^{\bsw})
	=
	\int_{D} K(\bsx,\KaKlScSuquark) \bbP_{\bsT}^{\bsw}(\mathrm d\bsx)
	=
	\sum_{k=0}^{n-1} w_{k} K(\bst_{k},\KaKlScSuquark) \in V_{\bsT} \subseteq \calH.
\end{align*}

\subsection{Kernel Cubature}

The basic idea of this paper is to view $h$ and $h_{\bsT}^{\bsw}$ as representing the integration operator and the cubature rule from \eqref{KaKlScSu:equ:empirical_mean_approximation} as elements in $\calH$---after all, by \eqref{KaKlScSu:equ:KME_reproducing_property}, 
$\mathbb E_{\mathbb P}[f]
=
\langle h , f \rangle_{\calH}$
and
$\bbE_{\mathbb P_{\bsT}^{\boldsymbol w}}[f]
=
\langle h_{\bsT}^{\boldsymbol w} , f \rangle_{\calH}$.
Hence, in order to reduce the approximation error in \eqref{KaKlScSu:equ:empirical_mean_approximation}, it seems natural to choose the weights $\bsw$ such that $h_{\bsT}^{\bsw}$ is the best approximation of $h$ in $\calH$, that is, its orthogonal projection onto $V_{\bsT}$.
It is well-known \cite[Lemma~10.24]{KaKlScSu:W2005} that such orthogonal projections within RKHSs correspond to interpolation:
\begin{lemma}
	\label{KaKlScSu:lemma:orthogonal_projection_coincides_with_interpolation_in_RKHS}
	Let $\bsT = (\bst_k)_{k=0}^{n-1} \in D^n$ be any point set in $D$.
	The orthogonal projection $\hat{g} = P_{V_{\bsT}} g$ of each $g \in \calH$ onto
	$V_{\bsT} = \mathrm{span}(K(\bst_0,\KaKlScSuquark), \ldots, K(\bst_{n-1},\KaKlScSuquark)) \subseteq \mathcal{H}$
	coincides with the unique solution of the following interpolation problem:
	find $\hat{g} \in V_{\bsT}$ such that
	$\hat{g}(\bst_{k})
	=
	% \stackrel{!}{=}
	g(\bst_{k})$
	for
	$k = 0,\dots,n-1$.
	In particular, $h_{\bsT}^{\bsw^{\ast}} = P_{V_{\bsT}} h$ is given by the unique solution $\bsw^{\ast}$ of
	\begin{equation}
		\label{KaKlScSu:eq:gramsystem}
		\calK_{\bsT} \bsw^{\ast} = (h(\bst_{k}))_{k=0}^{n-1},
	\end{equation}
	where $\calK_{\bsT} = (K(\bst_{k},\bst_{\ell}))_{k,\ell=0}^{n-1}$ is the Gram matrix.
\end{lemma}

\begin{definition}   
	We refer to the cubature weights $\bsw^{\ast}$ given by \eqref{KaKlScSu:eq:gramsystem} as \emph{optimal weights} and to the corresponding weighted cubature rule $Q_{\bsT}^{\bsw^{\ast}}$ as \emph{kernel cubature}.
\end{definition}

Another advantage of the embeddings $h,h_{\bsT}^{\bsw^{\ast}}$ is that the norm of their difference naturally describes the worst-case cubature error \eqref{KaKlScSu:equ:WCE_general_form_in_RKHS}.
Since the optimal weights $\bsw^{\ast}$ stem from an orthogonal projection, the Pythagorean theorem implies
\begin{equation}
	\label{KaKlScSu:equ:optimal_wce_formula}
	e(Q_{\bsT}^{\bsw^{\ast}} , \calH)^{2}
	=
	\| h - h_{\bsT}^{\bsw^{\ast}} \|_{\calH}^{2}
	=
	\| h \|_{\calH}^{2} -  \| h_{\bsT}^{\bsw^{\ast}} \|_{\calH}^{2}
	=
	\| h \|_{\calH}^{2} - (w^{\ast})^{\top} \calK_{\bsT} w^{\ast}.
\end{equation}
We will see in the next subsection that, for the kernels $K$ considered in this paper, $h = \mathbbm{1}$ turns out to be the constant unit function and \eqref{KaKlScSu:equ:optimal_wce_formula} reduces to $e(Q_{\bsT}^{\bsw} , \calH)^{2}
= 1 - \sum_{k=0}^{n-1} w_{k}^{\ast}$, making it easily computable once the weights are established.

\subsection{Specific Kernels and the Corresponding Sobolev Spaces}\label{KaKlScSu:subsec:sobolev}

In this paper, we consider the Sobolev space of dominating-mixed smoothness, see \cite[Section~2.4]{KaKlScSu:GSY2019} and \cite[Definition~1]{KaKlScSu:NS2023}.

\begin{definition}[Sobolev space of dominating mixed smoothness]%
	Let $s\in\mathbb Z_+$ and let $\boldsymbol{\gamma}=(\gamma_{\mathrm{\mathfrak{u}}})_{\mathrm{\mathfrak{u}}\subseteq\{1,\ldots,s\}}$ be a sequence of positive weights, termed \emph{coordinate weights}. The weighted Sobolev space $\calH^{\alpha}_{s,\boldsymbol{\gamma}}$ of order $\alpha \in\mathbb Z_+$ is a reproducing kernel Hilbert space with inner product\label{KaKlScSu:def:sov}
	{\fontsize{9}{11}
		% {\small
			\begin{align*}
				&\langle f , g \rangle _{\calH^{\alpha}_{s,\boldsymbol{\gamma}}}
				:=
				\left(\int_{[0,1]^s} f(\bsx)\rd \bsx\right)\;  \left(\int_{[0,1]^s} g(\bsx)\rd \bsx\right) +
				\\
				&\sum_{\emptyset\ne\mathrm{\mathfrak{u}}\subseteq\{1,\ldots,s\}}  \kern-5mm \gamma^{-1}_{\mathrm{\mathfrak{u}}}
				\sum_{ \substack{ \bstau \in \{0,\ldots,\alpha \}^{|\mathfrak{u}|} \\ \setv := \{j: \tau_j=\alpha \}} }  
				\int_{[\bszero_\setv,\bsone_\setv]} 
				\left(\int_{[\bszero_{-\setv},\bsone_{-\setv}]} \kern-8mm f^{(\bstau,\bszero_{-\mathrm{\mathfrak{u}}})} (\bsx) \rd \bsx_{-\setv} 
				\right)
				\left(\int_{[\bszero_{-\setv},\bsone_{-\setv}]}  \kern-8mm g^{(\bstau,\bszero_{-\mathrm{\mathfrak{u}}})} (\bsx) \rd \bsx_{-\setv} 
				\right)
				\kern-0.7mm
				\rd \bsx_\setv
				,
			\end{align*}
		}where for $\setv \subseteq \{1,\ldots,s\}$ we write $[\bsa_\setv,\bsb_\setv] := \prod_{j\in \setv} [a_j,b_j]$ and likewise for $-\setv := \{1,\ldots,s\} \setminus \setv$, $[\bsa_{-\setv},\bsb_{-\setv}]
		= \prod_{j \in -\setv} [a_j,b_j]
		:= \prod_{j \in \{1,\ldots,s\} \setminus \setv}
		[a_j,b_j]$.
		The reproducing kernel is given by
		\begin{align}\label{KaKlScSu:eq:kern}
			K^{\alpha}_{s,\boldsymbol{\gamma}} (\bsx,\bsy)
			&:=
			1+
			\sum_{\emptyset\ne\mathrm{\mathfrak{u}}\subseteq\{1,\ldots,s\}} \gamma_{\mathrm{\mathfrak{u}}}\prod_{j\in \mathrm{\mathfrak{u}}} \left( -1 + K^{\alpha}_{1,1}(x_j,y_j)\right)
			,
		\end{align} 
		with
		\begin{align*}
			K^{\alpha}_{1,1}(x_j,y_j)=
			1+\sum_{\tau=1}^\alpha \frac{B_{\tau}(x_j)}{\tau!} \, \frac{B_{\tau}(y_j)}{\tau!} + (-1)^{\alpha+1} \frac{\widetilde{B}_{2\alpha}(x_j-y_j)}{(2\alpha)!},
		\end{align*}
		where $\widetilde{B}_{2\alpha}$ is the $1$-periodic Bernoulli polynomial of order $2\alpha$, i.e., denoting by $\{x-y\}$ the fractional part of $x-y$ and using the standard Bernoulli polynomial $B_{2\alpha}$ we define
		\[
		\widetilde{B}_{2\alpha}(x-y):=B_{2\alpha}(\{x-y\}).
		\]
	\end{definition}
	
	\begin{remark}
		The term ``coordinate weights'' is used to help distinguish the weights $\boldsymbol{\gamma}$ from the cubature weights $\bsw$.
	\end{remark}
	
	We remark that, since Bernoulli polynomials $B_{\alpha}(x),\; \alpha\ge 1,$ integrate to $0$ over the unit interval $[0,1]$, we have
	\[
	\int_0 ^1 K^{\alpha}_{1,1}(x_j,\KaKlScSuquark) \, {\rm d} x_j 
	=
	\mathbbm{1},
	\quad
	\text{ hence,}
	\quad
	\int_{[0,1]^s} K^{\alpha}_{s,\boldsymbol{\gamma}} (\bsx,\KaKlScSuquark) \, {\rm d} \bsx
	=
	\mathbbm{1}.
	\]
	Due to this fact, we can represent the worst-case error \eqref{KaKlScSu:equ:WCE_general_form_in_RKHS} by
	\begin{align}
		\label{KaKlScSu:eq:wce_general}
		\begin{split}
			e(Q_{\bsT}^{\bsw} , \calH^{\alpha}_{s,\boldsymbol{\gamma}})^{2}
			&=
			\langle \mathbbm{1} , \mathbbm{1} \rangle_{\calH^{\alpha}_{s,\boldsymbol{\gamma}}}
			-
			2 \sum_{k=0}^{n-1}w_k \langle \mathbbm{1} , K^{\alpha}_{s,\boldsymbol{\gamma}}(\bst_k,\KaKlScSuquark) \rangle_{\calH^{\alpha}_{s,\boldsymbol{\gamma}}}
			\\
			&\hspace{1em}
			+
			\sum_{k=0}^{n-1}\sum_{k'=0}^{n-1} w_{k}w_{k'} \langle K^{\alpha}_{s,\boldsymbol{\gamma}} (\bst_k,\KaKlScSuquark) , K^{\alpha}_{s,\boldsymbol{\gamma}} (\bst_{k'},\KaKlScSuquark) \rangle_{\calH^{\alpha}_{s,\boldsymbol{\gamma}}}
			\\
			&=
			1-2\sum_{k=0}^{n-1}w_k+\sum_{k=0}^{n-1}\sum_{k'=0}^{n-1} w_{k}w_{k'} K^{\alpha}_{s,\boldsymbol{\gamma}} (\bst_k,\bst_{k'}),
		\end{split}
	\end{align}
	which for equal and optimal weights reduces to
	\begin{equation}\label{KaKlScSu:eq:wce-weights}
		e(Q_{\bsT}^{\bsw} , \calH^{\alpha}_{s,\boldsymbol{\gamma}})^{2}
		=
		\begin{cases}
			-1 + n^{-2} \sum_{k,k'=0}^{n-1} K^{\alpha}_{s,\boldsymbol{\gamma}} (\bst_k,\bst_{k'})
			&
			\text{if } \bsw = (n^{-1})_{k=0}^{n-1},
			\\
			1 - \sum_{k=0}^{n-1} w_{k}^{\ast}
			&
			\text{if } \bsw = \bsw^{\ast}.
		\end{cases}
	\end{equation}

	\subsection{Lattice Points and the Tent Transform}
	\label{KaKlScSu:subsec:lattice_rules_tent_transform}
	
	Let $\tilde{\bsT} = (\tilde{\bst}_k)_{k=0}^{n-1} \subset [0,1]^s$ be a lattice point set defined by
	\[
	\tilde{\bst}_k = \left\{ \frac{k \boldsymbol{z}}{n} \right\}, \quad k = 0, \ldots, n - 1,
	\]
	where $\{\KaKlScSuquark\}$ denotes the componentwise fractional part and $\boldsymbol{z}$ is the so-called \emph{generating vector}, which consists of $s$ elements of integers, $\boldsymbol{z}\in\{1, \ldots, n-1\}^s$.
	In Section~\ref{KaKlScSu:sec:numerics}, our cubature rules $Q_{\bsT}$ and $Q_{\bsT}^{\bsw}$ will be based on a point set $\bsT = (\bst_k)_{k=0}^{n-1} \subset [0,1]^s$ defined in one of the following ways:
	\begin{itemize}
		\item
		as the unshifted lattice point set itself, $\bsT = \tilde{\bsT}$;
		\item
		as a randomly shifted lattice point set, defined by $\bst_k = \{ \tilde{\bst}_k + \Delta \}$, where $\Delta$ is a fixed random shift sampled uniformly from $[0,1]^s$;
		\item
		or as a randomly shifted and tent-transformed lattice point set, defined by $\bst_k = \phi(\{ \tilde{\bst}_k + \Delta \})$, where the so-called \emph{baker's transform} is applied componentwise:
		\[
		\phi(\boldsymbol{t}) := (\phi(t_1), \ldots, \phi(t_s)), \quad
		\phi(t) := 1 - |2t - 1|, \quad t \in \mathbb{R}.
		\]
	\end{itemize}
	We refer to \cite{KaKlScSu:DKP2022,KaKlScSu:DKS2013,KaKlScSu:N1992,KaKlScSu:SJ1994} for the general theory of numerical integration using lattice rules, and to \cite{KaKlScSu:DNP2014,KaKlScSu:GSY2019,KaKlScSu:H2002} for the use of tent-transformed lattice rules in the context of non-periodic functions.
	Our motivation for using tent-transformed lattice rules stems from their ability to achieve second-order convergence in Sobolev spaces of smoothness $\alpha = 2$; see {\cite[Corollary~1]{KaKlScSu:GSY2019}}.

	\section{Theoretical Considerations}
	\label{KaKlScSu:sec:theory}
	
	Classical QMC theory typically bounds the approximation error by the inequality
	\[
	|I f - Q_{\bsT}^{\bsw} f|
	\leq
	e(Q_{\bsT}^{\bsw}, \calH)\, \| f \|_{\calH},
	\qquad
	f \in \calH.
	\]
	Since the second factor $\| f \|_{\calH}$ is constant with respect to $n$ for fixed $f$ (often normalized to one for simplicity), the error is ultimately controlled by the worst-case error $e(Q_{\bsT}^{\bsw}, \calH)$.
	While $\bsw^{\ast}$ minimizes the worst-case error among all possible weights,
	\[
	e(Q_{\bsT}^{\bsw^{\ast}}, \calH)
	\leq
	e(Q_{\bsT}^{\bsw}, \calH)
	\qquad
	\text{for all }
	\bsw \in \bbR^{n},
	\]
	its construction via an orthogonal projection further improves this bound by replacing the constant factor $\| f \|_{\calH}$ with the distance $\mathrm{dist}_{\calH}(f,V_{\bsT})$, which can be expected to decrease as $n$ grows and $V_{\bsT}$ increasingly approximates $\calH$:
	
	\begin{proposition}
		\label{KaKlScSu:prop:approximation_error_additional_gain}
		Let $\bsT = (\bst_k)_{k=0}^{n-1} \in D^n$ be any point set in $D$ and $\calH$ be an RKHS with kernel $K\colon D\times D \to \mathbb{R}$.
		Let
		$V_{\bsT} = \mathrm{span}(K(\bst_0,\KaKlScSuquark), \ldots, K(\bst_{n-1},\KaKlScSuquark)) \subseteq \mathcal{H}$
		and let $\bsw^{\ast}$ satisfy \eqref{KaKlScSu:eq:gramsystem}.
		Then
		\begin{equation*}
			|If -  Q_{\bsT}^{\bsw^{\ast}} f|
			\leq
			e(Q_{\bsT}^{\bsw^{\ast}}, \calH)\,
			\mathrm{dist}_{\calH}(f,V_{\bsT}),
			% \| h - P_{V_{\bsT}} h \|_{\calH} \, \| f - P_{V_{\bsT}} f \|_{\calH}.
		\end{equation*}
		where $\mathrm{dist}_{\calH}(f,V_{\bsT})$ denotes the distance between $f$ and $V_{\bsT}$ in $\calH$.
	\end{proposition}
	
	\begin{proof}
		Since $h_{\bsT}^{\bsw^{\ast}} = P_{V_{\bsT}} h$, we obtain $h-h_{\bsT}^{\bsw^{\ast}} \perp V_{\bsT}$ and, by the Cauchy--Schwarz inequality
		\begin{align*}
			|If -  Q_{\bsT}^{\bsw^{\ast}} f|
			&=
			|\langle h-h_{\bsT}^{\bsw^{\ast}} , f \rangle_{\calH}|
			\\
			&=
			|\langle h-h_{\bsT}^{\bsw^{\ast}} , f  - P_{V_{\bsT}} f \rangle_{\calH}|
			\\
			&\leq
			\| h-h_{\bsT}^{\bsw^{\ast}} \|_{\calH} \, \| f - P_{V_{\bsT}} f \|_{\calH}
			\\
			&=
			e(Q_{\bsT}^{\bsw^{\ast}}, \calH)\, \mathrm{dist}_{\calH}(f,V_{\bsT}),
		\end{align*}
		proving the claim.\qed
	\end{proof}
	
	\begin{remark}
		We noted above that the distance $\mathrm{dist}_{\calH}(f, V_{\bsT})$ can be expected to decrease as $n$ increases. This is intuitive, as the space $V_{\bsT}$ becomes richer and more capable of approximating elements in $\calH$. However, to the best of our knowledge, no theoretical result is currently available that quantifies the rate of this convergence across arbitrary RKHSs. Quantitative rates are known in specific cases, for example for periodic Sobolev spaces and certain shift-invariant kernels, but do not appear to extend directly to the non-periodic Sobolev spaces of dominating mixed smoothness considered in this paper. There appear to be connections to Gaussian process regression on lattices, which is an active area of research \cite{KaKlScSu:Osborne2025Convergence,KaKlScSu:Teckentrup2020ConvergenceGP}.
	\end{remark}

	\subsection{Rate Doubling in the One-Dimensional Setting}
	
	In this subsection, we investigate the effect of using optimal weights for an (unshifted) lattice rule in dimension $s=1$, which simply corresponds to a left-Riemann rule, that is,
	$$
	t_k=\frac{k}{n},\quad k=0,\ldots,n-1.
	$$
	If $f\in \mathcal H_{1,\mathbf 1}^1$, then it is a consequence of~\eqref{KaKlScSu:eq:wce-weights} and the identity $\sum_{k,k'=0}^{n-1}K_{1,\mathbf 1}^1(\frac{k}{n},\frac{k'}{n})=\frac{3n^2+1}{3}$ that the equally weighted quadrature rule $Q_{\boldsymbol T} f =\frac1n\sum_{k=0}^{n-1}f(t_k)$ admits the error rate
	$$
	|I f - Q_{\boldsymbol T} f|=\mathcal O(n^{-1}).
	$$
	
	Defining the sequence of weights $(w_k^*)_{k=0}^{n-1}$ as the solution to the system~\eqref{KaKlScSu:eq:gramsystem} with $K_{1,\mathbf 1}^1(x,y)=1+\frac12B_2(|x-y|)+(x-1/2)(y-1/2)$, $x,y\in[0,1]$, denoting the one-dimensional kernel corresponding to $\mathcal H_{1,\mathbf 1}^1$, we obtain
	\[
	w_0^*=\frac{1}{2n}\frac{12n^3}{12n^3+n+3},
	\qquad
	w_k^*=2w_0^*,\quad k\in\{1,\ldots,n-2\},
	\qquad
	w_{n-1}^*=3w_0^*.
	\]
	In this special case, the optimally weighted quadrature rule $Q_{\boldsymbol T}^{\bsw^{\ast}} f = \sum_{k=0}^{n-1} w_{k}^{\ast} f(t_k)$ exhibits a quadratic error rate:
	\begin{lemma}\label{KaKlScSu:lemma:doubled}
		Suppose that $f\in \mathcal H_{1,\mathbf 1}^2$. 
		Then 
		$$
		|I f - Q_{\boldsymbol T}^{\boldsymbol w^*} f |=\mathcal O(n^{-2}).
		$$
	\end{lemma}
	\begin{proof}
		The quadrature error can be recast as 
		\begin{align*}
			\int_0^1 f(y)\,{\rm d}y-\sum_{k=0}^{n-1}w_k^*f(t_k)&=\int_0^1\langle f,K_{1,\mathbf 1}^1(\KaKlScSuquark,y)\rangle_{\mathcal H_{1,\mathbf 1}^1}\,{\rm d}y-\sum_{k=0}^{n-1}w_k^* \langle f,K_{1,\mathbf 1}^1(\KaKlScSuquark,t_k)\rangle_{\mathcal H_{1,\mathbf 1}^1}\\
			&=\bigg\langle f,\int_0^1 K_{1,\mathbf 1}^1(\KaKlScSuquark,y)\,{\rm d}y-\sum_{k=0}^{n-1}w_k^*K_{1,\mathbf 1}^1(\KaKlScSuquark,t_k)\bigg\rangle_{\mathcal H_{1,\mathbf 1}^1}\\
			&=\langle f,h-h_{\bsT}^{\bsw^{\ast}}\rangle_{\mathcal H_{1,\mathbf 1}^1},
		\end{align*}
		where $h:=\int_0^1 K_{1,\mathbf 1}^1(\KaKlScSuquark,y)\,{\rm d}y=\mathbbm{1}$ and $h_{\bsT}^{\bsw^{\ast}}:=\sum_{k=0}^{n-1}w_k^*K_{1,\mathbf 1}^1(\KaKlScSuquark,t_k)$ is the kernel interpolant of~$h$. Using the definition of the inner product in the weighted Sobolev space of smoothness $\alpha=1$ we have
		\begin{align*}
			\langle f,h-h_{\bsT}^{\bsw^{\ast}}\rangle_{\mathcal H_{1,\mathbf 1}^1}&=\bigg(\int_0^1f(y)\,{\rm d}y\bigg)\bigg(\int_0^1 (h(y)-h_{\bsT}^{\bsw^{\ast}}(y))\,{\rm d}y\bigg)\\
			&\quad +\int_0^1f'(y)(h'(y)-(h_{\bsT}^{\bsw^{\ast}})'(y))\,{\rm d}y.
		\end{align*}
		Using integration by parts, the absolute value of the latter integral can be estimated as
		\begin{align*}
			&
			\bigg| \int_0^1f'(y)(h'(y)-(h_{\bsT}^{\bsw^{\ast}})'(y))\,{\rm d}y \bigg|\\
			&=
			\bigg| f'(1)(h(1)-h_{\bsT}^{\bsw^{\ast}}(1))-f'(0)(h(0)-h_{\bsT}^{\bsw^{\ast}}(0))-\int_0^1 f''(y)(h(y)-h_{\bsT}^{\bsw^{\ast}}(y))\,{\rm d}y \bigg|\\
			&\leq \frac{6n}{12n^3+n+3} |f'(1)| +\|f''\|_{L^2(0,1)}\|h-h_{\bsT}^{\bsw^{\ast}}\|_{L^2(0,1)},
		\end{align*}
		where we used the fact that $h(0)-h_{\boldsymbol T}^{\boldsymbol w^*}(0)=0$, a consequence of the fact that $h_{\boldsymbol T}^{\boldsymbol w^*}$ interpolates $h$ at the lattice point $t_0=0$, as well as the identity
		\begin{align*}
			h(1)-h_{\boldsymbol T}^{\boldsymbol w^*}(1)&=1-\frac{12n^2}{12n^3+n+3}\bigg(\frac12 K_{1,\mathbf 1}^1(1,0)+\frac32 K_{1,\mathbf 1}^1\bigg(1,\frac{n-1}{n}\bigg)+\sum_{k=1}^{n-2}K_{1,\mathbf 1}^1\bigg(1,\frac{k}{n}\bigg)\bigg)\\
			&=1-\frac{12n^2}{12n^3+n+3}\bigg(\frac{29}{12} + \frac{3}{4n^2}-\frac{3}{2n}+\sum_{k=1}^{n-2}\frac{5n^2+3k^2}{6n^2}\bigg)\\
			&=1-\frac{12n^2}{12n^3+n+3}\bigg(n-\frac{5}{12n}+\frac{1}{4n^2}\bigg)\\
			&=\frac{6n}{12n^3+n+3},
		\end{align*}
		where we used $K_{1,\mathbf 1}^1\big(1,\frac{k}{n}\big)=\frac56+\frac{k^2}{2n^2}$, $k\in\{0,\ldots,n-1\}$, and $\sum_{k=1}^{n-2}k^2=\frac{(2n-3)(n-1)(n-2)}{6}$. Therefore
		$$
		|\langle f,h-h_{\bsT}^{\bsw^{\ast}}\rangle_{\mathcal H_{1,\mathbf 1}^1}|\leq \frac{6n}{12n^3+n+3}|f'(1)| + \|h-h_{\bsT}^{\bsw^{\ast}}\|_{L^2(0,1)}\big(\|f\|_{L^2(0,1)}+\|f''\|_{L^2(0,1)}\big).
		$$
		Since $\frac{6n}{12n^3+n+3}|f'(1)|=\mathcal O(n^{-2})$, it remains to assess the convergence rate of $\|h-h_{\bsT}^{\bsw^{\ast}}\|_{L^2(0,1)}$.
		Making use of the identities
		\begin{align*}
			&\sum_{k=0}^{n-1}w_k^*=\frac{24n^2}{12n^3+n+3}+\frac{12n^2(n-2)}{12n^3+n+3},\\
			&\sum_{k=0}^{n-1}\sum_{\ell=0}^{n-1}w_k^*w_\ell^* \bigg(\frac{46}{45}-\frac16 t_k^2+\frac{1}{12}t_k^3-\frac{1}{24}t_k^4+\frac14 t_k^2t_\ell -\frac16 t_\ell^2\\
			&\quad\quad\quad+\frac14 t_kt_\ell^2-\frac14 t_k^2t_\ell^2+\frac{1}{12}t_\ell^3-\frac{1}{24}t_\ell^4+\frac{1}{12}|t_k-t_\ell|^3\bigg)=\frac{720n^6+n^2+60n-45}{5(12n^3+n+3)^2},
		\end{align*}
		where the latter identity is valid for $n\geq 2$, we obtain, for $n\geq 2$,
		\begin{align*}
			\|h-&h_{\bsT}^{\bsw^{\ast}}\|_{L^2(0,1)}^2
			=
			\int_0^1 \bigg(1-\sum_{k=0}^{n-1}w_k^*K_{1,\mathbf 1}^1(x,t_k)\bigg)^2\,{\rm d}x\\
			&=\int_0^1\bigg(1-2\sum_{k=0}^{n-1}w_k^*K_{1,\mathbf 1}^1(x,t_k)+\sum_{k=0}^{n-1}\sum_{\ell=0}^{n-1}w_k^*w_\ell^* K_{1,\mathbf 1}^1(x,t_k)K_{1,\mathbf1}^1(x,t_\ell)\bigg)\,{\rm d}x\\
			&=1-2\sum_{k=0}^{n-1}w_k^*+\sum_{k=0}^{n-1}\sum_{\ell=0}^{n-1}w_k^*w_\ell^* \bigg(\frac{46}{45}-\frac16 t_k^2+\frac{1}{12}t_k^3-\frac{1}{24}t_k^4+\frac14 t_k^2t_\ell -\frac16 t_\ell^2\\
			&\quad\quad\!\!+\frac14 t_kt_\ell^2-\frac14 t_k^2t_\ell^2+\frac{1}{12}t_\ell^3-\frac{1}{24}t_\ell^4+\frac{1}{12}|t_k-t_\ell|^3\bigg)
			\\
			&=\frac{6n(n+15)}{5(12n^3+n+3)^2}.%\vspace*{-.1cm}
		\end{align*}
		Hence, $|If-Q_{\bsT}^{\boldsymbol w^*}f|=\mathcal O(n^{-2})$ as claimed.\qed
	\end{proof}
	Lemma~\ref{KaKlScSu:lemma:doubled} illustrates that using kernel cubature can improve the error convergence rate of a lattice rule in one dimension.

	\section{Numerical Experiments}
	\label{KaKlScSu:sec:numerics}
	
	We briefly describe the construction of the Gramian matrix $\mathcal K_{\boldsymbol T}$ for common types of coordinate weights and analyze the computational complexity in~Section~\ref{KaKlScSu:sec:assembly}. We then investigate the behavior of kernel cubature applied to QMC point sets in two experiments. In Section~\ref{KaKlScSu:sec:pdenumex}, we compare the performance of equally weighted lattice rules against kernel cubature constructed for the same point sets, applied to an elliptic partial differential equation (PDE) with a parametric input coefficient, while Section~\ref{KaKlScSu:subsec:numerics_WCE} compares the worst-case errors of equally weighted and optimally weighted QMC point sets in unweighted Sobolev spaces $\mathcal H_{s,\mathbf 1}^{\alpha}$ with varying smoothness parameters $\alpha$.\newpage

	\subsection{Construction of the Gramian Matrix and Its Computational Complexity}\label{KaKlScSu:sec:assembly}
	
	In the context of information-based complexity and uncertainty quantification, it is often the case to consider the dominant cost to be the $n$ evaluations of the integrand $f$, see e.g., \cite{KaKlScSu:LMRS2021,KaKlScSu:U2017}. 
	By contrast, the construction of the lattice rule and the computation of the associated weights are typically negligible in comparison and can often be performed ``offline''.
	
	That said, for completeness, we provide here a brief analysis of the computational cost associated with constructing the cubature weights, which involves two main steps:  
	(1) assembling the Gramian matrix $\calK_{\bsT}$, and  
	(2) solving the resulting linear system.
	The second step, solving the resulting linear system, has a computational complexity of $\mathcal{O}(n^3)$ using standard direct methods such as Cholesky decomposition, which is appropriate here given that the kernel matrix is symmetric and positive definite.
	We now turn our attention to the first step: computing entries of the kernel matrix, i.e., evaluating $K_{s,\bsgamma}^{\alpha}(\bsx, \bsy)$ for given $\bsx, \bsy \in [0,1]^s$.
	Recall that
	$$
	K_{s,\bsgamma}^{\alpha}(\bsx,\bsy)=\sum_{\setu\subseteq\{1,\ldots,s\}}\gamma_{\setu}\prod_{j\in\setu}\eta_{\alpha}(x_j,y_j),\label{KaKlScSu:eq:generickernel}
	$$
	where $\eta_{\alpha}(x,y)=\sum_{\tau=1}^{\alpha}\frac{1}{(\tau!)^2}B_{\tau}(x)B_{\tau}(y)+\frac{(-1)^{\alpha+1}}{(2\alpha)!}\widetilde B_{2\alpha}(x-y)$. 
	The cost of evaluating this kernel depends on the structure of the coordinate weights $\bsgamma=(\gamma_{\setu})_{\setu\subseteq\{1,\ldots,s\}}$. We refer to~\cite[Section~5.2]{KaKlScSu:KKKNS2022} for details and provide a brief overview below. In what follows, it is assumed that evaluating $\eta_\alpha$ has a constant cost. We also use the convention that a product over an empty set is defined to be equal to 1.
	\begin{itemize}
		\item \emph{Product weights} $\gamma_{\setu}=\prod_{j\in\setu}\widetilde\gamma_j$ are specified by a sequence of nonnegative numbers $(\widetilde\gamma_j)_{j=1}^s$, and the kernel can be equivalently written as
		$$
		K_{s,\bsgamma}^{\alpha}(\bsx,\bsy)=\prod_{j=1}^s (1+\widetilde\gamma_j\eta_\alpha(x_j,y_j)).
		$$
		This expression can be evaluated in $\mathcal O(s)$ time for one pair $(\bsx,\bsy)$ and, for any point set $\boldsymbol T=(\boldsymbol t_k)_{k=0}^{n-1}$ in $[0,1]^s$, the matrix $\mathcal K_{\boldsymbol T}=(K_{s,\boldsymbol\gamma}^{\alpha}(\bst_k,\bst_\ell))_{k,\ell=0}^{n-1}$ can be assembled in $\mathcal O(sn^2)$ time.
		\item \emph{Product-and-order dependent (POD) weights} $\gamma_{\setu}=\Gamma_{|\setu|}\prod_{j\in\setu}\widetilde\gamma_j$ are specified by two sequences of nonnegative numbers $(\Gamma_k)_{k=0}^s$ and $(\widetilde\gamma_j)_{j=1}^s$, and the kernel can be equivalently written as
		\begin{align}
			K_{s,\bsgamma}^{\alpha}(\bsx,\bsy)=\sum_{\ell=0}^s \Gamma_{\ell}P_{s,\ell},\label{KaKlScSu:eq:podkern}
		\end{align}
		where the sequence $(P_{k,\ell})_{k,\ell=0}^s$ can be computed recursively by
		\begin{align*}
			&P_{k,0}=1\quad\text{for all}~k\in\{0,\ldots,s\},\\
			&P_{k,\ell}=0\quad\text{for all}~k\in\{0,\ldots,s\}~\text{and}~\ell\in\{k+1,\ldots,s\},\\
			&P_{k,\ell}=P_{k-1,\ell}+\widetilde\gamma_k \eta_\alpha(x_k,y_k)P_{k-1,\ell-1}~\text{for all}~k\in\{1,\ldots,s\}~\text{and}~\ell\in\{1,\ldots,k\}.
		\end{align*}
		The cost to obtain $K_{s,\bsgamma}^\alpha(\bsx,\bsy)$ using the expression~\eqref{KaKlScSu:eq:podkern} is $\mathcal O(s^2)$ for one pair $(\bsx,\bsy)$  and, for any point set $\boldsymbol T=(\boldsymbol t_k)_{k=0}^{n-1}$ in $[0,1]^s$, the matrix $\mathcal K_{\boldsymbol T}=(K_{s,\boldsymbol\gamma}^{\alpha}(\bst_k,\bst_\ell))_{k,\ell=0}^{n-1}$ can be assembled in $\mathcal O(s^2n^2)$ time.
	\end{itemize}
	Since the system matrix $\mathcal K_{\boldsymbol T}$ is symmetric, it has at most  $\frac{n(n+1)}{2}$ unique entries that need to be constructed. The cost of solving the matrix equation~\eqref{KaKlScSu:eq:gramsystem} is $\mathcal O(n^3)$ independently of the weights $\bsgamma$ and dimension $s$.
	
	It is possible to obtain recurrence formulas for other classes of coordinate weights: for example, using {\em smoothness-driven product-and-order dependent (SPOD)} weights with smoothness degree $\sigma\in\mathbb N$, the cost to obtain $K_{s,\boldsymbol\gamma}^{\alpha}(\boldsymbol x,\boldsymbol y)$ is $\mathcal O(s^2\sigma^2)$ for one pair $(\boldsymbol x,\boldsymbol y)$ and, for any point set $\boldsymbol T=(\boldsymbol t_k)_{k=0}^{n-1}$ in $[0,1]^s$, the matrix $\mathcal K_{\boldsymbol T}=(K_{s,\boldsymbol\gamma}^{\alpha}(\bst_k,\bst_\ell))_{k,\ell=0}^{n-1}$ can be assembled in $\mathcal O(s^2\sigma^2n^2)$ time. For details, we refer to~\cite[Section~5.2]{KaKlScSu:KKKNS2022}.

	\subsection{Application to PDE Uncertainty Quantification}
	\label{KaKlScSu:sec:pdenumex}
	
	Let $\Omega=(0,1)^2$. We consider the elliptic PDE
	\begin{align}
		\begin{cases}
			-\nabla \KaKlScSuquark (a(\boldsymbol x,\boldsymbol y)\nabla u(\boldsymbol x,\boldsymbol y))=f(\boldsymbol x),&\boldsymbol x\in \Omega,~\boldsymbol y\in[-\tfrac12,\tfrac12]^s,\\
			u(\boldsymbol x,\boldsymbol y)=0,&\boldsymbol x\in\partial \Omega,~\boldsymbol y\in[-\tfrac12,\tfrac12]^s,
		\end{cases}\label{KaKlScSu:eq:pde}
	\end{align}
	equipped with the parametric diffusion coefficient
	\begin{align}
		\label{KaKlScSu:eq:pde2}
		a(\boldsymbol x,\boldsymbol y)=\frac12+\frac12\sum_{j=1}^s j^{-2}y_j\sin(j\pi x_{1})\sin(j\pi x_{2})
	\end{align}
	for $\boldsymbol x=(x_1,x_2)\in \Omega$ and $\boldsymbol y=(y_1,\ldots,y_s)\in[-\tfrac12,\tfrac12]^s$, where $s$ is referred to as the truncation dimension. 
	In this case, it can be shown (cf.\ \cite{KaKlScSu:kss12}) that, using the POD coordinate weights
	\begin{align}\label{KaKlScSu:eq:podweights}
		\gamma_{\setu}=\bigg(|\setu|!\prod_{j\in\setu}\frac{b_j}{\sqrt{{2\zeta(2\lambda)}/{(2\pi^2)^\lambda}}}\bigg)^{\frac{2}{1+\lambda}}\quad \text{for all}~\setu\subseteq\{1,\ldots,s\},
	\end{align}
	with $b_j=(1-\frac12\zeta(2))^{-1} j^{-2}$ and $\lambda=\frac{1}{2-2\delta}$, $\delta=0.05$, randomly shifted rank-1 lattice rules achieve \emph{dimension-independent} convergence rates for the root-mean-square integration error of the expected value
	\begin{align}
		\mathbb E[G(u)]=\int_{[-\frac12,\frac12]^s}G(u(\KaKlScSuquark,\boldsymbol y))\,{\mathrm d}\boldsymbol y,\label{KaKlScSu:eq:expectedvalue}
	\end{align}
	where $G\!:H_0^1(\Omega)\to\mathbb R$ is an arbitrary bounded linear functional (the quantity of interest).
	The PDE~\eqref{KaKlScSu:eq:pde} was discretized using a first-order finite element method with mesh width $h=2^{-5}$.

	\begin{figure}[!t]
		\begin{center}
			\includegraphics[height=.468\textwidth]{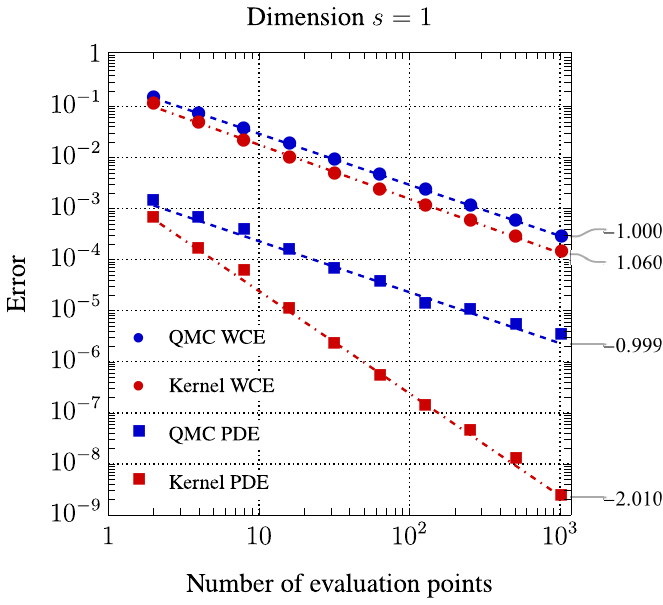}~\includegraphics[height=.468\textwidth,trim=.75cm 0cm 0cm 0cm,clip]{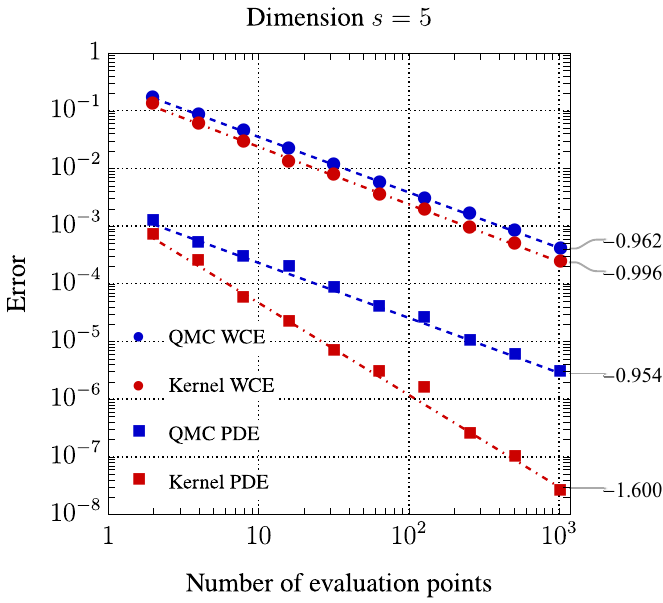}\\[.25cm]
			\includegraphics[height=.468\textwidth]{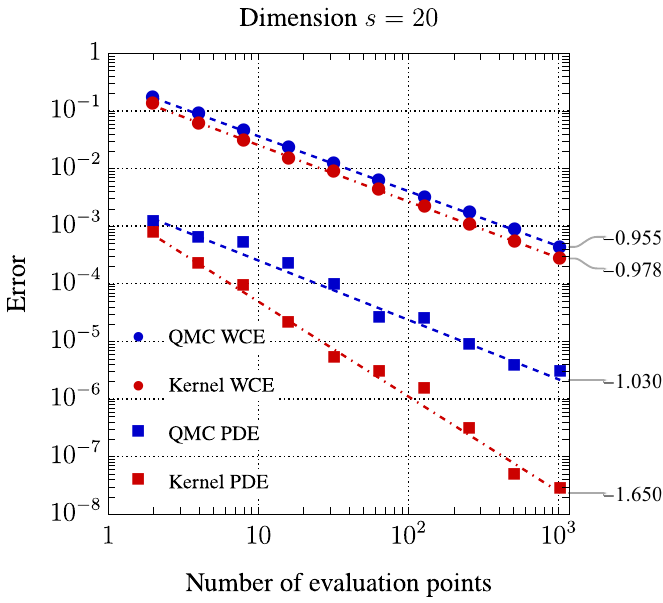}~\includegraphics[height=.468\textwidth,trim=.75cm 0cm 0cm 0cm,clip]{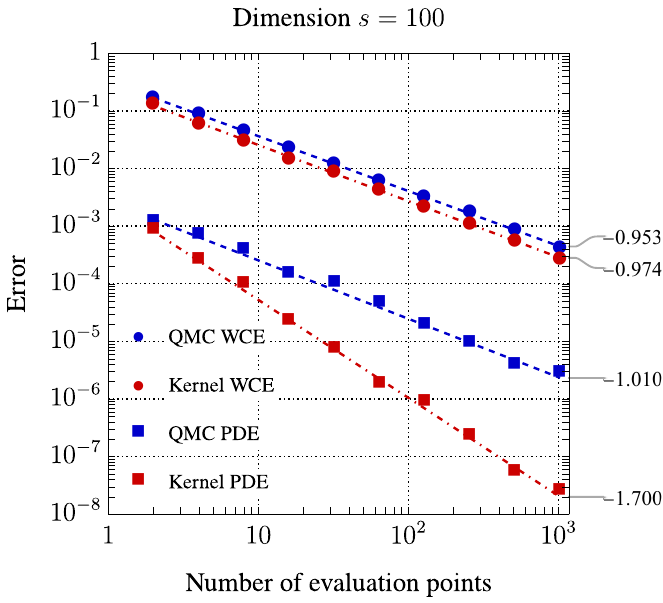}
		\end{center}
		\caption{The cubature errors for the PDE example~\eqref{KaKlScSu:eq:pde}--\eqref{KaKlScSu:eq:pde2} were computed using both equally weighted rank-1 lattice rules (``QMC PDE'') and kernel cubature rules (``Kernel PDE''), with both methods evaluated on the same lattice point sets. We used $s\in\{1,5,20,100\}$ as the truncation dimensions. The weights for the kernel cubature were obtained subject to $\mathcal H_{s,\bsgamma}^1$ and we also illustrate the computed worst-case errors for both equally weighted lattice rules (``QMC WCE'') and the corresponding kernel cubature rules (``Kernel WCE'') in the space $\mathcal H_{s,\bsgamma}^1$. The cubature errors have been averaged over $R=8$ random shifts.}
		\label{KaKlScSu:fig:numex}
	\end{figure}

	We set $s\in\{1,5,20,100\}$ as the truncation dimension in \eqref{KaKlScSu:eq:pde2}, choose $f(\boldsymbol x)=x_1$ as the source term, and set $G(v)=\int_{\Omega}v(\boldsymbol x)\,{\mathrm d}\boldsymbol x$ as the quantity of interest. We used the fast component-by-component algorithm~\cite{KaKlScSu:qmc4pde} to find an extensible generating vector for $n=2^k$, $k=1,\ldots,12$, corresponding to the POD coordinate weights~\eqref{KaKlScSu:eq:podweights}. We approximated the expected value~\eqref{KaKlScSu:eq:expectedvalue} using two methods based on the same fixed lattice rule:\medskip
	\begin{enumerate}
		\item Equally weighted lattice QMC rules for $n=2^k$, $k=1,\ldots,10$;
		\item Kernel cubature rules over the same lattice points for $n=2^k$, $k=1,\ldots,10$.
	\end{enumerate}\medskip%
	For error estimation, we applied $R=8$ random shifts to each point set. As the reference solution, we used the numerical approximations corresponding to an $n=2^{12}$ point QMC rule for experiment 1 and an $n=2^{12}$ point kernel cubature point set for experiment 2. To obtain the weights $\boldsymbol w^*$ of the randomly shifted kernel cubature rules, we solved the linear system~\eqref{KaKlScSu:eq:gramsystem} for each randomly shifted lattice point set corresponding to the kernel~\eqref{KaKlScSu:eq:kern} with $\alpha=1$ and POD coordinate weights~\eqref{KaKlScSu:eq:podweights}.
	In this case, it is necessary to assemble the elements of the Gram matrix $\mathcal K_{\bsT}$ recursively; see Section~\ref{KaKlScSu:sec:assembly} for details. In addition, we computed the worst-case errors for both methods using the square root of the expression~\eqref{KaKlScSu:eq:wce-weights}. The results are displayed in Figure~\ref{KaKlScSu:fig:numex}.

	While the kernel cubature does not improve the essentially linear worst-case error rate of the underlying rank-1 lattice point set, the numerical results seem to indicate that the kernel cubature rate is significantly better. For dimension $s=1$, the kernel cubature rate for the PDE problem is double that of the equally weighted cubature rule (as indicated by Lemma~\ref{KaKlScSu:lemma:doubled}) while for increased dimensions $s$ the observed cubature convergence rate lies approximately between $-1.6$ and $-1.7$. This improvement may be attributed to the error decomposition presented in Proposition~\ref{KaKlScSu:prop:approximation_error_additional_gain}: if the distance between $f$ and $V_{\boldsymbol T}$ is decreasing as $n$ grows, then the observed kernel cubature rate can exceed the worst-case error rate by a significant amount.

	\subsection{Worst-Case Errors for Sobolev Spaces of Varying Smoothness}
	\label{KaKlScSu:subsec:numerics_WCE}
	
	In this subsection, we demonstrate improved convergence rates of lattice rules with optimized weights by directly calculating worst-case errors.
	We first optimize the weight for $\calH^2_{s,\bsone}$ and calculate the worst-case error.
	We consider shifted and tent-transformed lattice rules $\bsT = \phi(\{\tilde{\bsT} + \Delta\})$ as described in Section~\ref{KaKlScSu:subsec:lattice_rules_tent_transform}, where a single random shift $\Delta$ is added to the lattice before the transformation, preventing points from coinciding after the transform is applied (which would render the Gram matrix $\calK_{\bsT}$ singular).
	It is shown in \cite[Corollary~1]{KaKlScSu:GSY2019} that second order convergence can be achieved by a tent-transformed lattice rule in $\calH^{2}_{s,\boldsymbol{\gamma}}$.
	In this numerical example we show that the optimally weighted tent-transformed lattice rule can achieve even faster convergence in $\calH^{4}_{2,\bsone}$. 
	\begin{figure}[!h]
		\begin{center}
			\includegraphics[width=.69\textwidth]{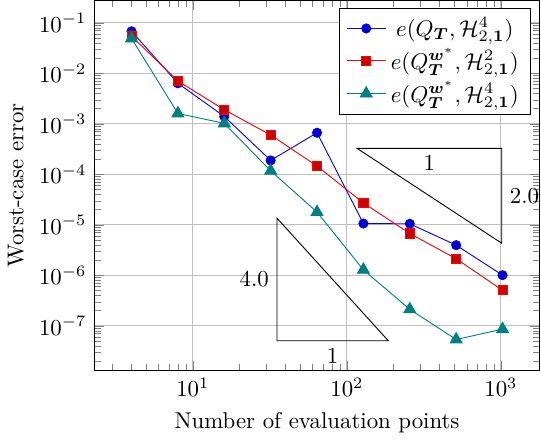}
			\caption{Worst-case errors for tent-transformed lattice rules. Here, $Q_{\bsT}$ and $Q_{\bsT}^{\bsw^*}$ denote lattice rules based on the same shifted, tent-transformed lattice, using equal weights and optimized weights, respectively. The weights $\bsw^*$ are optimized for the second-order Sobolev space $\calH^2_{2,\bsone}$. While $Q_{\bsT}$ is expected to achieve second-order convergence, the observation that $Q_{\bsT}^{\bsw^*}$ attains a convergence rate higher than two appears to be previously unreported.}\label{KaKlScSu:fig:wce1}
		\end{center}
	\end{figure}
	Figure~\ref{KaKlScSu:fig:wce1} shows the two-dimensional case.
	We use the generating vector $\bsz=(1,182667)$ for $n=2^2,\ldots,2^{10}$, from \cite{KaKlScSu:CKN2006} (available online at Frances Kuo's webpage \cite{KaKlScSu:FrancesPage}). 
	We remind the reader that the worst-case errors can be easily calculated by~\eqref{KaKlScSu:eq:wce_general} and~\eqref{KaKlScSu:eq:wce-weights}, i.e.,
	\begin{align*}
		e(Q_{\bsT}^{\bsw^*} , \calH^{\alpha}_{s,\boldsymbol{\gamma}})^{2}
		&=
		1-\sum_{k=0}^{n-1} w_k^{\ast},
		\\
		e(Q_{\bsT} , \calH^{\alpha}_{s,\boldsymbol{\gamma}})^{2}
		&=
		-1+n^{-2}\kern-2mm\sum_{k,k'=0} ^{n-1} K^{\alpha}_{s,\boldsymbol{\gamma}} (\bst_k,\bst_{k'}),
	\end{align*}
	and
	\[
	e(Q_{\bsT}^{\bsw^*} , \calH^{2\alpha}_{s,\boldsymbol{\gamma}})^{2}
	=
	1-2\sum_{k=0}^{n-1} w_k^{\ast}+\sum_{k=0}^{n-1}\sum_{k'=0}^{n-1} w^{\ast}_{k}w^{\ast}_{k'} K^{2\alpha}_{s,\boldsymbol{\gamma}} (\bst_k,\bst_{k'}),
	\]
	where in the last line, the weights $\bsw^{\ast}$ are optimized for $ \calH^{\alpha}_{s,\boldsymbol{\gamma}}$, not $ \calH^{2\alpha}_{s,\boldsymbol{\gamma}}$.
	
	We observe that a convergence rate faster than second order is attained.
	Note that there are other ways to achieve convergence faster than $\mathcal{O}(n^{-2})$ using lattice rules for non-periodic functions. One example is the \emph{periodization strategy}, which uses a change of variables in order to obtain a periodic integrand from a non-periodic one. However, it is not known how to avoid the curse of dimensionality when using this strategy, see \cite{KaKlScSu:KSW2007}. Another example is the symmetrized lattice rule \cite[Corollary~2]{KaKlScSu:GSY2019}, but this is also cursed by dimensionality because the required number of points grows exponentially with the dimension.

	\section{Concluding Remarks}
	\label{KaKlScSu:sec:conclusion}
	
	In this paper, we introduced a weighted version of QMC cubature, referred to as \emph{kernel cubature}, where the weights are chosen to minimize the distance between the kernel mean embeddings of the true and empirical distributions.  
	We provided a theoretical result (Proposition~\ref{KaKlScSu:prop:approximation_error_additional_gain}) suggesting an improved convergence rate for kernel cubature compared to the equally weighted case, proved a corresponding statement for dimension $s=1$, and presented numerical results in dimensions $s = 1, 5, 20,$ and $100$, focusing on lattice rules and an elliptic PDE problem with a random coefficient.  
	Additionally, we explored the behavior of the worst-case error when weights optimized for a Sobolev space of dominating-mixed smoothness $\alpha$ were applied to a space of higher smoothness. In dimension $s=2$, this led to a significant acceleration in the convergence rate.
	
	While this work established several theoretical insights, its primary focus was experimental.
	Future research will aim to establish improved convergence guarantees in arbitrary dimensions, as well as investigate constructions of lattices and other QMC point sets specifically designed for kernel cubature.

	\section*{Acknowledgments}
	IK and CS were funded by the Deutsche Forschungsgemeinschaft (DFG, German Research Foundation) under Germany's Excellence Strategy (EXC-2046/1, project 390685689) through project EF1-19 of the Berlin Mathematics Research Center MATH+.
	This work of YS was supported by the Research Council of Finland (decisions 348503 and 359181). The work of VK was supported by the Research Council of Finland (Flagship of Advanced Mathematics for Sensing, Imaging and Modelling grant 359183).
	We thank Frances Kuo, Fred Hickernell and Robert Gruhlke for helpful collegial discussions.

\end{document}